\documentclass[12pt]{article}

\usepackage{amssymb}
\usepackage{amsfonts}
\usepackage{eurosym}
\usepackage{amsmath}

\allowdisplaybreaks
\newtheorem{theorem}{Theorem}

\newtheorem{example}[theorem]{Example}

\newtheorem{lemma}[theorem]{Lemma}

\newtheorem{proposition}[theorem]{Proposition}
\newtheorem{remark}[theorem]{Remark}

\newenvironment{proof}[1][Proof]{\noindent\textbf{#1.} }{\ \rule{0.5em}{0.5em}}

\begin{document}

\title{An Osgood's criterion for a semilinear stochastic differential
equation}
\author{Jorge A. Le\'{o}n\thanks{%
Partially supported by a CONACyT grant.}, Liliana Peralta\thanks{%
Partially supported by a CONACyT fellowship.} \\
\texttt{jleon@ctrl.cinvestav.mx, lperalta@ctrl.cinvestav.mx} \\
Departamento de Control Autom\'{a}tico\\
Cinvestav-IPN, Apartado Postal 14-740, 07000\\
M\'{e}xico D.F., Mexico \and Jos\'{e} Villa-Morales\thanks{%
Partially supported by the grant PIM14-4 of UAA.} \\
\texttt{\ jvilla@correo.uaa.mx} \\
Departamento de Matem\'{a}ticas y F\'{\i}sica\\
Universidad Aut\'{o}noma de Aguascalientes\\
Av. Universidad 940, C.P. 20131\\
Aguascalientes, Ags., Mexico}
\date{}
\maketitle

\begin{abstract}
The purpose of this paper is to give an Osgood's criterion for solutions of
semilinear stochastic differential equations of the form $X_{t}=\xi
+\int_{0}^{t}b(s,X_{s})ds+\int_{0}^{t}\sigma (s)X_{s}dW_{s},\ t\geq 0$.
Here, $b$ is a non-negative, non-decreasing by components and continuous
random field and $\sigma $ is a predictable and continuous process. Also we
present a generalization of the so-called Feller's test whenever $\sigma
\equiv 1$.

\bigskip


\end{abstract}

\section{Introduction}

In the case of an ordinary differential equation (ODE), the explosion in
finite time is a very old and well-known subject. In fact, in 1898, W.F.
Osgood (see \cite{os}) established that the solution $y$ of 
\begin{eqnarray}
y^{\prime }(t) &=&b(y(t)),\ \ t>0,  \label{eqdo} \\
y(0) &=&\xi ,  \notag
\end{eqnarray}%
with $b>0$, blows up in finite time if and only if 
\begin{equation}
\int_{\xi }^{\infty }\frac{ds}{b(s)}<\infty .  \label{texposg}
\end{equation}%
Moreover, (\ref{texposg}) is the explosion time of the equation (\ref{eqdo}).

For the case of partial differential equations (PDE) the study of the
phenomenon of explosion was originated with the works of S. Kaplan \cite%
{kaplan} and H. Fujita \cite{fujita} and it is currently an area of very
fruitful research, see for instance \cite{hu}, \cite{pevilla}.

On the other hand, within the stochastic framework, since Feller's test \cite%
{Feller} for determining the explosion time of the autonomous stochastic
differential equation (SDE) 
\begin{eqnarray}
dX_{t} &=&b(X_{t})dt+\sigma (X_{t})dW_{t},\ \ t>0,  \label{sdea} \\
X_{0} &=&\xi ,  \notag
\end{eqnarray}
only few results have been developed, see for instance \cite{lionjose}.
However, the growing interest in the application of the explosion of SDE has
motivated its study. For example, when $b$ and $\sigma $ are power functions
the equation (\ref{sdea}) can be used for modeling the crack failure of some
materials (see for example \cite{sobczyk}). Also some numerical schemes have
been analyzed in order to approximate the time of explosion (consult D\'{a}%
vila et al. \cite{davila}).

For the non-autonomous case, Feller's test and Osgood's criterion are not
useful anymore. The contribution of this work is to deal with an extension
of the Osgood's criterion for the blow up in finite time of the solution $X$
of the semilinear stochastic differential equation%
\begin{eqnarray}
dX_{t} &=&b(t,X_{t})dt+\sigma (t)X_{t}dW_{t},\ \ t>0,  \label{sdeno} \\
X_{0} &=&\xi .  \notag
\end{eqnarray}

This paper is organized as follows. In Section \ref{Sec:pre} we precise the
Osgood's criterion and then present some necessary results in order to prove
an extension of Osgood's criterion for non-autonomous equations. The
criterion of blow up in finite time for the equation (\ref{sdeno}) is stated
and proved in Section \ref{Sec:main}.

\section{Preliminaries}

\label{Sec:pre} In this section for the convenience of the reader we briefly
introduce three well-known topics: the Feller's test, a comparison lemma and
the Osgood's criterion. We also present an extension of the latter.

Conditions necessary to determine whether or not the solution $X$ of an
equation as (\ref{sdea}) explodes in finite time with probability 1 have
been developed, most notably by William Feller \cite{Feller} (with $\xi$ a real number). The Feller's
explosion test just needs to know the coefficients $b$ and $\sigma $ of the
equation.

\begin{theorem}[Feller's test]
\label{feller} Suppose that $b,\sigma :(\ell ,r)\rightarrow \mathbb{R}$,
with $-\infty \leq \ell <r\leq \infty $, are continuous functions and $%
\sigma ^{2}>0$ in $(\ell ,r)$. The explosion time $\tau $ of the solution $X$
of equation (\ref{sdea}) is finite with probability 1 if and only if one of
the following conditions holds:
\begin{description}
\item[$(i)$] $v(r-)<\infty $ and $v(\ell +)<\infty $,

\item[$(ii)$] $v(r-)<\infty $ and $p(\ell +)=-\infty $, or

\item[$(iii)$] $v(\ell +)<\infty $ and $p(r-)=\infty $,
\end{description}
where 
\begin{eqnarray}
p(x) &=&\int_{\zeta }^{x}\exp \left( -2\int_{\zeta }^{s}\frac{b(r)dr}{\sigma
^{2}(r)}\right) ds,  \label{funfeller} \\
v(x) &=&\int_{\zeta }^{x}p^{\prime }(y)\int_{\zeta }^{y}\frac{2dz}{p^{\prime
}(z)\sigma ^{2}(z)}dy,  \label{defvF}
\end{eqnarray}
here $\zeta \in (\ell,r)$ is a constant.
\end{theorem}

\begin{proof}
See in Chapter 5 of \cite{karatzas} the Proposition 5.32.\hfill
\end{proof}

\bigskip

In the case where the coefficient $b$ is non-negative and the term $\sigma $
is zero, the equation (\ref{sdea}) becomes an ordinary differential equation
for which the criterion of explosion is known as Osgood's criterion. To
establish this result we introduce the following function 
\begin{equation*}
B_{\xi }(x)=\int_{\xi }^{x}\frac{ds}{b(s)},\ \ \xi \leq x\leq \infty .
\end{equation*}

\begin{proposition}[Osgood's criterion]
\label{osgood}Let $b:{\mathbb{R}}\rightarrow {\mathbb{R}}$ be a continuous
function such that $b>0$ in $(c,\infty )$ and $\xi >c\in {\mathbb{R}}$. If $y
$ is a solution of the integral equation 
\begin{equation*}
y(t)=\xi +\int_{0}^{t}b(y(s))ds,\ \ t\geq 0,
\end{equation*}%
then the explosion time $T_{e}:=\sup \{t>0:|y(t)|<\infty \}$ of $y$ is
finite if and only if $B_{\xi }(\infty )<\infty $. Moreover the solution
must be 
\begin{equation*}
y(t)=B_{\xi }^{-1}(t),\ \ 0\leq t<B_{\xi }(\infty )=T_{e}.
\end{equation*}
\end{proposition}

\bigskip

We use the following notation, if $X$ is the solution of certain equation
with initial condition $x_{0}$, then by $T_{x_{0}}^{X}$ we will denote the
time of explosion of $X$.

The following comparing result will be essential in our study of the
behavior of semilinear SDE.

\begin{lemma}
\label{lem:comp} Let $b:{\mathbb{R}}\rightarrow {\mathbb{R}}$ be a
continuous non-negative function. Also assume that $b$ is non-decreasing and
positive in $(c,\infty )$, $\xi >c\in {\mathbb{R}}$ and $x,y:[0,T]%
\rightarrow {\mathbb{R}}$ are two continuous functions:

\begin{description}
\item[$(i)$] If 
\begin{equation}
y(t)\geq \xi +\int_{0}^{t}b(y(s))ds,\ \ t\in \lbrack 0,T],  \label{eq:comp}
\end{equation}%
and 
\begin{equation*}
x(t)=\xi +\int_{0}^{t}b(x(s))ds,\ \ t\in \lbrack 0,T],
\end{equation*}%
then $y(t)\geq x(t)$, for all $t\in \lbrack 0,T]$.

\item[$(ii)$] Moreover, if we assume $y>c$ on $\lbrack 0,T]$, 
\begin{equation*}
y(t)\leq \xi +\int_{0}^{t}b(y(s))ds,\ \ t\in \lbrack 0,T],
\end{equation*}%
and 
\begin{equation*}
x(t)=\xi +\int_{0}^{t}b(x(s))ds,\ \ t\in \lbrack 0,T],
\end{equation*}%
then $y(t)\leq x(t)$, for all $t\in \lbrack 0,T]$.
\end{description}
\end{lemma}

\begin{proof}
\textbf{Case (i):} Let $0<r<\xi -c$ and consider the solution $x_{r}$ of%
\begin{equation*}
x_{r}(t)=\xi -r+\int_{0}^{t}b(x_{r}(s))ds,\ \ t\in \lbrack 0,T].
\end{equation*}%
Define $N_{r}$ as the set $\{t\in \lbrack 0,T]:x_{r}(s)\leq y(s),\forall
s\in \lbrack 0,t]\}$. Observe that $N_{r}\neq \varnothing $ (indeed $0\in
N_{r}$). Since $b\geq 0$, then $x_{r}\geq \xi -r$ on $\lbrack 0,T]$. Using
that $b$ is non-decreasing in $(c,\infty )$ we have that $b(y(s))\geq
b(x_{r}(s))$ if $y(s)\geq x_{r}(s)>c$. Let us see that $T_{r}:=\sup
N_{r}<T$ is not possible. In fact, since%
\begin{eqnarray*}
\lim_{\varepsilon \downarrow 0}\left( y(T_{r}+\varepsilon
)-x_{r}(T_{r}+\varepsilon )\right) &\geq
&r+\int_{0}^{T_{r}}[b(y(s))-b(x_{r}(s))]ds \\
&&+\lim_{\varepsilon \downarrow 0}\int_{T_{r}}^{T_{r}+\varepsilon
}[b(y(s))-b(x_{r}(s))]ds \\
&\geq &r>0
\end{eqnarray*}%
then the continuity of $y-x_{r}$ implies $T_{r}<\sup N_{r}$. Therefore $%
x_{r}\leq y$ on $\lbrack 0,T]$. On the other hand, by Proposition \ref%
{osgood} we have 
\begin{equation*}
B_{\xi }^{-1}(t)=\lim_{r\downarrow 0}B_{\xi -r}^{-1}(t)=\lim_{r\downarrow
0}x_{r}(t)\leq y(t),\ \ t\in \lbrack 0,T].
\end{equation*}%
To get the first equality, in the above expression, we have used that the
continuity of $B_{\cdot }(t)$ implies the continuity of $B_{\cdot }^{-1}(t)$.

\textbf{Case (ii):} The proof is like the previous case, but now it is
convenient consider the solution $x_{r}$ of%
\begin{equation*}
x_{r}(t)=\xi +r+\int_{0}^{t}b(x_{r}(s))ds,\ \ t\in \lbrack 0,T],
\end{equation*}%
where $r>0$.\hfill 
\end{proof}

\bigskip

In Theorem 2.2.4 of \cite{pachpa} the interested reader can see other
version of Lemma \ref{lem:comp} (in \cite{pachpa} the function $b$ is
supposed to be monotone throughout its domain).

Using the last two results we can state the following extension of Osgood's
criterion for non-autonomous equations, which is important in itself.

\begin{proposition}
Let $b:[0,\infty )\times \mathbb{R}\rightarrow \mathbb{R}$ be a non-negative
continuous function. We also assume that $b$ is positive and non-decreasing
by components on $[0,\infty )\times (c,\infty ),$ with $c\in \mathbb{R}$.
Then, a solution of the equation%
\begin{eqnarray}
y^{\prime }(t) &=&b(t,y(t)),\ \ t>0,  \label{eq_propo} \\
y(0) &=&\xi ,  \notag
\end{eqnarray}%
where $\xi >c$, explodes in finite time if and only if 
\begin{equation*}
\int_{\xi }^{\infty }\frac{ds}{b(a,s)}<\infty ,
\end{equation*}%
for some $a>0$.
\end{proposition}

\begin{proof}
Suppose that $y$ is a solution of (\ref{eq_propo}) and that explodes at time $T_{\xi
}^{y}<\infty $. Inasmuch as $b\left( \cdot ,x\right) $ is non-decreasing we
obtain%
\begin{equation*}
y(t)\leq \xi +\int_{0}^{t}b(T_{\xi }^{y},y(s))ds,\ \ t<T_{\xi }^{y}.
\end{equation*}%
Using that $b\geq 0$ we see, from (\ref{eq_propo}), that $y$ is
non-decreasing. Therefore $y(0)=\xi $ imply $y>c$ on $[0,\infty )$. Hence
applying Lemma \ref{lem:comp} $(ii)$ we can deduce that the solution of the
equation%
\begin{equation*}
v(t)=\xi +\int_{0}^{t}b(T_{\xi }^{y},v(s))ds,\ \ t\geq 0,
\end{equation*}%
explodes in finite time. Thus, by Osgood's criterion we can conclude that $%
\int_{\xi }^{\infty }(b(T_{\xi }^{y},s))^{-1}ds<\infty$.

Reciprocally, suppose that the solution $y$ of (\ref{eq_propo}) does not
blow up in finite time. Let $a>0$, using that $b\geq 0$ we have
\begin{equation*}
y(t)\geq \xi +\int_{a}^{t}b(s,y(s))ds\geq \xi +\int_{a}^{t}b(a,y(s))ds,\ \
t\geq a.
\end{equation*}%
Consequently, the solution of 
\begin{equation*}
u(t)=\xi +\int_{0}^{t}b\left( a,u(s)\right) ds,\ \ t\geq 0,
\end{equation*}%
does not explode in finite time, because by Lemma \ref{lem:comp} $(i)$ we
can deduce that $u(t)\leq y(t+a)$, for each $t\geq 0$. Therefore $\int_{\xi
}^{\infty }(b(a,s))^{-1}ds=\infty $ due to Osgood's criterion.\hfill
\end{proof}

\section{Semilinear stochastic differential equations}

\label{Sec:main}The purpose of this section is to study the semilinear
stochastic differential equation (SDE)%
\begin{equation}
X_{t}=\xi +\int_{0}^{t}b(s,X_{s})ds+\int_{0}^{t}\sigma (s)X_{s}dW_{s},\ \
t\geq 0.  \label{eqpmb}
\end{equation}%
Hereinafter $W=\{W_{t}:t\geq 0\}$ is a one-dimensional Brownian motion
defined on a filtered probability space $(\Omega ,\mathcal{F},(\mathcal{F}%
_{t})_{t\geq 0},\mathbb{P})$, where $(\Omega ,\mathcal{F},\mathbb{P})$ is
complete and $(\mathcal{F}_{t})_{t\geq 0}$ is supposed to satisfy the usual
conditions. The initial condition $\xi $ is a $\mathcal{F}_{0}$-measurable
random variable and the coefficients $b$ and $\sigma $ satisfy the following
assumptions:

\begin{description}
\item[\textbf{H1}:] $b:(\Omega \times \lbrack 0,\infty )\times \mathbb{R},%
\mathcal{P}\otimes \mathcal{B}(\mathbb{R}))\rightarrow \mathbb{R}$ is a
continuous non-negative random field with probability one. Here $\mathcal{P}$
is the predictable $\sigma $-algebra and $\mathcal{B}(\mathbb{R})$ is the
Borel $\sigma $-algebra on $\mathbb{R}.$

\item[\textbf{H2}:] $\sigma :\Omega \times \lbrack 0,\infty )\rightarrow 
\mathbb{R}$ is a predictable and continuous process.
\end{description}

Note that these assumptions, on the coefficients in (\ref{eqpmb}), will be
assumed on the rest of the section.

\subsection{A particular case of Feller's test}

As a first step, in order to have a better understanding of the remainder of
the article, in this subsection we analyze the autonomous case of the
equation (\ref{eqpmb}) when $\sigma \equiv 1$, specifically we study the SDE%
\begin{equation}
Z_{t}=\xi +\int_{0}^{t}b(Z_{s})ds+\int_{0}^{t}Z_{s}dW_{s},\ \ t\geq 0,
\label{pacauto}
\end{equation}%
where $\xi >0$ is a real number and $b:{\mathbb{R}}\rightarrow \mathbb{R}$
is a continuous non-negative function. We begin with the next result for
which we use the following notation 
\begin{equation*}
\bar{b}(x)=\frac{b(x)}{x},\ \ x>0.
\end{equation*}

\begin{theorem}
\label{VillaFeller}Suppose that $\bar{b}:(0,\infty )\rightarrow {\mathbb{R}}$
is a non-decreasing function such that $\bar{b}>1/2$. Then, the explosion
time $T_{\xi }^{Z}$ of the solution $Z$ of $(\ref{pacauto})$ is finite with
probability $1$ if and only if 
\begin{equation*}
\int_{\xi }^{\infty }\frac{ds}{2b(s)-s}<\infty .
\end{equation*}
\end{theorem}

\begin{proof}
Applying It\^{o}'s formula, to the process $R$ defined as 
\begin{equation*}
R_{t}=Z_{t}\exp \left( -W_{t}+\frac{t}{2}\right) ,\ \ t<T_{\xi }^{Z},
\end{equation*}%
we obtain%
\begin{equation}
R_{t}=\xi +\int_{0}^{t}e^{-W_{s}+\frac{s}{2}}b(e^{W_{s}-\frac{s}{2}%
}R_{s})ds,\ \ t<T_{\xi }^{Z}.  \label{auxito}
\end{equation}%
If one use that $b$ is non-negative then one see that%
\begin{equation}
Z_{t}=R_{t}\exp \left( W_{t}-\frac{t}{2}\right)>0,\ \ t\geq 0.
\label{seuspeje}
\end{equation}%
Then we will be able to prove the result using the Feller's test for
explosions with $l=0$, $r=\infty $ and $\zeta =\xi $ (see the case $(ii)$ of
Theorem \ref{feller}). The monotonicity of $\bar{b}$ turns out%
\begin{eqnarray*}
p(0) &=&-\int_{0}^{\xi }\exp \left( 2\int_{s}^{\xi }\bar{b}(r)\frac{dr}{r}%
\right) ds \\
&\leq &-\int_{0}^{\xi }\exp \left( 2\bar{b}(s)\int_{s}^{\xi }\frac{dr}{r}%
\right) ds \\
&=&-\int_{0}^{\xi }\left( \frac{\xi }{s}\right) ^{2\bar{b}(s)}ds.
\end{eqnarray*}%
Using that $\xi /s\geq 1$ and $2\bar{b}(s)-1>0$ we have%
\begin{equation*}
p(0)\leq -\int_{0}^{\xi }\frac{\xi }{s}ds=-\infty .
\end{equation*}%
Therefore from Feller's test, it is enough to show that $v(\infty )<\infty $
if and only if $\int_{\xi }^{\infty }(2b(s)-s)^{-1}ds<\infty $. By the
definition (\ref{defvF}) of $v$ we deduce%
\begin{eqnarray}
v(\infty ) &=&\int_{\xi }^{\infty }\int_{\xi }^{y}\frac{2}{z^{2}}\exp \left(
-2\int_{z}^{y}\frac{\bar{b}(r)}{r}dr\right) dzdy  \notag \\
&\geq &\int_{\xi }^{\infty }\int_{\xi }^{y}\frac{2}{z^{2}}\exp \left( -2\bar{%
b}(y)\int_{z}^{y}\frac{dr}{r}\right) dzdy  \notag \\
&=&2\int_{\xi }^{\infty }y^{-2\bar{b}(y)}\int_{\xi }^{y}\frac{dz}{z^{2-2\bar{%
b}(y)}}dy  \notag \\
&=&2\int_{\xi }^{\infty }\left\{ \frac{1}{2b(y)-y}-\frac{\xi ^{2\bar{b}%
(y)-1}y^{1-2\bar{b}(y)}}{2b(y)-y}\right\} dy.  \label{densv}
\end{eqnarray}%
Since the function $2\bar{b}-1$ is non-decreasing we obtain%
\begin{eqnarray*}
\int_{\xi }^{\infty }\frac{\xi ^{2\bar{b}(y)-1}y^{1-2\bar{b}(y)}}{2b(y)-y}dy
&=&\xi ^{-1}\int_{\xi }^{\infty }\frac{1}{2\bar{b}(y)-1}\left( \frac{\xi }{y}%
\right) ^{2\bar{b}(y)}dy \\
&\leq &\frac{\xi ^{-1}}{2\bar{b}(\xi )-1}\int_{\xi }^{\infty }\left( \frac{%
\xi }{y}\right) ^{2\bar{b}(y)}dy.
\end{eqnarray*}%
Hence, the facts that $\xi /y<1$ and $\bar{b}$ is non-decreasing lead us to%
\begin{eqnarray*}
\int_{\xi }^{\infty }\frac{\xi ^{2\bar{b}(y)-1}y^{1-2\bar{b}(y)}}{2b(y)-y}dy
&\leq &\frac{\xi ^{-1}}{2\bar{b}(\xi )-1}\int_{\xi }^{\infty }\left( \frac{%
\xi }{y}\right) ^{2\bar{b}(\xi )}dy \\
&=&\frac{1}{(2\bar{b}(\xi )-1)^{2}}.
\end{eqnarray*}%
Thus, (\ref{densv}) implies that $\int_{\xi }^{\infty
}(2b(s)-s)^{-1}ds<\infty $ if $v(\infty )<\infty $.

On the other hand, by Fubini's theorem 
\begin{eqnarray*}
v(\infty ) &\leq &\int_{\xi }^{\infty }\int_{\xi }^{y}\frac{2}{z^{2}}\exp
\left( -2\bar{b}(z)\int_{z}^{y}\frac{dr}{r}\right) dzdy \\
&=&\int_{\xi }^{\infty }\int_{\xi }^{y}\frac{2}{z^{2}}\left( \frac{y}{z}%
\right) ^{-2\bar{b}(z)}dzdy \\
&=&2\int_{\xi }^{\infty }z^{2\bar{b}(z)-2}\int_{z}^{\infty }y^{-2\bar{b}%
(z)}dydz \\
&=&2\int_{\xi }^{\infty }\frac{z^{-1}}{2\bar{b}(z)-1}dz \\
&=&2\int_{\xi }^{\infty }\frac{dz}{2b(z)-z}.
\end{eqnarray*}%
Consequently, $v(\infty )<\infty $ if $\int_{\xi }^{\infty
}(2b(s)-s)^{-1}dz<\infty $.\hfill
\end{proof}

\ 

\subsection{A generalization of Feller's test}

Now we deal with the non-autonomous stochastic differential equation%
\begin{equation}
Y_{t}=\xi +\int_{0}^{t}b(s,Y_{s})ds+\int_{0}^{t}Y_{s}dW_{s},\ \ t\geq 0,
\label{eq:sigma1}
\end{equation}%
where $\xi $ is defined as in equation (\ref{eqpmb}) and remember that the
function $b$ satisfies the condition \textbf{H1}. Henceforth we will use the
notation 
\begin{equation*}
\tilde{b}(\omega ,t,x)=\frac{b(\omega ,t,e^{x})}{e^{x}},\ \ (\omega ,t,x)\in
\Omega \times \lbrack 0,\infty )\times \mathbb{R}.
\end{equation*}

\begin{theorem}
\label{trespuntodos} Let $c\geq 0$ and suppose that with probability one the
function $\tilde{b}:{\Omega \times \lbrack }0,\infty )\times {\mathbb{R}}%
\rightarrow {\mathbb{R}}$ satisfy
\newline
$(i)$ for each $x\in \mathbb{R}$, $\tilde{b}(\cdot ,x):{\Omega \times
\lbrack }0,\infty )\rightarrow {\mathbb{R}}$ is non-decreasing (in the time
component),$\newline
(ii)$ for each $t\in \lbrack 0,\infty )$, $\tilde{b}(t,\cdot ):{\Omega
\times (c,\infty )}\rightarrow {\mathbb{R}}$ is non-decreasing (in the space
component),$\newline
(iii)$ for each $(t,x)\in \lbrack 0,\infty )\times {\mathbb{R}}$, $\tilde{b}%
(t,x)\geq 1/2$ and for each $(t,x)\in \lbrack 0,\infty )\times (c,\infty )$, 
$\tilde{b}(t,x)>1/2$. \newline
Then, for almost all $\omega _{0}$ in 
\begin{eqnarray*}
{\tilde{\Omega}}=\{\omega \in \Omega  &:&W_{\cdot }(\omega )\ \text{is
continuous and }b(\omega ,\cdot ),\ \tilde{b}(\omega ,\cdot )\text{ satisfy} \\
&&\text{the above hypotheses, }\xi (\omega )>0\},
\end{eqnarray*}%
the solution $Y_{\cdot }(\omega _{0})$ of equation (\ref{eq:sigma1})
explodes in finite time if and only if 
\begin{equation}
\int_{\theta }^{\infty }\frac{ds}{2b(\omega _{0},a,s)-s}<\infty ,\ \ \forall
\theta >e^{c},  \label{conexpmfea}
\end{equation}%
for some $a>0$.
\end{theorem}

\begin{remark}
$(a)$ Note that $a$ depends on $\omega _{0}.$\newline
$(b)$ If with probability one $\tilde{b}(t,x)>1/2$ for each $(t,x)\in
\lbrack 0,\infty )\times {\mathbb{R}}$, then (\ref{conexpmfea}) is
equivalent to%
\begin{equation*}
\int_{\xi }^{\infty }\frac{ds}{2b(\omega _{0},a,s)-s}<\infty .
\end{equation*}
\end{remark}

\begin{proof}[Proof of Theorem \protect\ref{trespuntodos}]
Applying It\^{o}'s formula as in (\ref{auxito}) to 
\begin{equation}
R_{t}=Y_{t}\exp \left( -W_{t}+\frac{t}{2}\right) ,\ \ 0\leq t<T_{\xi }^{Y},
\label{otraaucex}
\end{equation}%
and using that $b\geq 0$ we obtain a non-decreasing process $R$ given by%
\begin{equation}
R_{t}=\xi +\int_{0}^{t}e^{-W_{s}+\frac{s}{2}}b(s,e^{W_{s}-\frac{s}{2}%
}R_{s})ds,\ \ 0\leq t<T_{\xi }^{Y}.  \label{auxva}
\end{equation}%
Then (\ref{otraaucex}) implies that $Y>0$, thus the process 
\begin{equation}
Z_{t}=\log (Y_{t}),\ \ 0\leq t<T_{\xi }^{Y},  \label{defZ}
\end{equation}%
is well defined and $T_{\log \xi }^{Z}=T_{\xi }^{Y}$. We can apply again It%
\^{o}'s formula to obtain%
\begin{equation}
Z_{t}=\log \xi +\int_{0}^{t}\tilde{b}(s,Z_{s})ds+W_{t}-\frac{t}{2},\ \ 0\leq
t<T_{\log \xi }^{Z}.  \label{eqauxl}
\end{equation}

Now fix $\omega _{0}\in {\tilde{\Omega}}$, for which the expression (\ref{eqauxl}) is satisfied.

\textit{Necessity}: Let us suppose that $T_{\xi }^{Y}(\omega _{0})<\infty $.
Because $Y(\omega _{0})>0$, then $Y(\omega _{0})$ explodes to $+\infty $,
hence $T_{\log \xi }^{Z}(\omega _{0})<\infty $ and $Z(\omega _{0})$ explodes
to $+\infty $. Therefore we can find a $T\in (0,T_{\log \xi }^{Z}(\omega
_{0}))$, such that 
\begin{equation*}
Z_{t}(\omega _{0})>c,\ \ \forall t\in \lbrack T,T_{\log \xi }^{Z}(\omega
_{0})].
\end{equation*}%
We rewrite equation (\ref{eqauxl}) as%
\begin{multline*}
Z_{T+t}(\omega _{0})=Z_{T}(\omega _{0})+W_{T+t}(\omega _{0})-W_{T}(\omega
_{0})+\int_{T}^{T+t}\left\{ \tilde{b}(\omega _{0},s,Z_{s})-\frac{1}{2}%
\right\} ds, \\
0\leq t<T_{\log \xi }^{Z}(\omega _{0})-T.
\end{multline*}%
Setting $y(t)=Z_{T+t}(\omega _{0})$ we have%
\begin{multline*}
y(t)=Z_{T}(\omega _{0})+W_{T+t}(\omega _{0})-W_{T}(\omega
_{0})+\int_{0}^{t}\left\{ \tilde{b}(\omega _{0},T+s,y(s))-\frac{1}{2}%
\right\} ds, \\
0\leq t<T_{\log \xi }^{Z}(\omega _{0})-T.
\end{multline*}%
The fact that $\tilde{b}$ is non-decreasing in the time variable bring about
the inequality%
\begin{equation*}
y(t)\leq M+\int_{0}^{t}\left\{ \tilde{b}(\omega _{0},T_{\log \xi
}^{Z}(\omega _{0}),y(s))-\frac{1}{2}\right\} ds,\ \ 0\leq t<T_{\log \xi
}^{Z}(\omega _{0})-T,
\end{equation*}%
with 
\begin{equation*}
M=Z_{T}(\omega _{0})+2\sup \left\{ |W_{t}(\omega _{0})|:t\in \lbrack
0,T_{\log \xi }^{Z}(\omega _{0})]\right\} +c.
\end{equation*}%
Consider the integral equation%
\begin{equation*}
x(t)=M+\int_{0}^{t}\left\{ \tilde{b}(\omega _{0},T_{\log \xi }^{Z}(\omega
_{0}),x(s))-\frac{1}{2}\right\} ds,\ \ t\geq 0.
\end{equation*}%
Lemma \ref{lem:comp} $(ii)$ yields $T_{M}^{x}\leq T_{\log \xi }^{Z}(\omega
_{0})$. Since $M>c$ the Proposition \ref{osgood} implies 
\begin{equation*}
2\int_{M}^{\infty }\frac{ds}{2\tilde{b}(\omega _{0},T_{\log \xi }^{Z}(\omega
_{0}),s)-1}<\infty 
\end{equation*}%
and therefore the continuity of $b$ gives 
\begin{equation*}
\int_{\theta }^{\infty }\frac{ds}{2b(\omega _{0},T_{\log \xi }^{Z}(\omega
_{0}),s)-s}<\infty ,\ \ \forall \theta >e^{c}.
\end{equation*}

\textit{Sufficiency}: Now let us assume $T_{\xi }^{Y}(\omega _{0})=\infty $
and take $a>0$ fix. As before, $Y_{t}(\omega _{0})>0$, for each $t\geq 0$,
and (\ref{defZ}) turns out, $T_{\log \xi }^{Z}(\omega _{0})=\infty $. By the
law of iterated logarithm (see for instance Theorem 4.3 in Le\'{o}n and
Villa \cite{lionjose}) we can find a sequence $\{t_{n}:n\in {\mathbb{N}}\}$
such that $a\leq t_{n}\uparrow \infty $ and 
\begin{equation}
\inf \{W_{h+t_{n}}(\omega _{0}):0\leq h\leq 1\}\uparrow \infty ,\ \
n\rightarrow \infty .  \label{neq:suc}
\end{equation}%
Hence from equation (\ref{eqauxl}) and using the hypothesis that $\tilde{b}$
satisfies, we obtain%
\begin{equation}
Z_{t+t_{n}}(\omega _{0})\geq \tilde{m}_{n}+\int_{0}^{t}\left\{ \tilde{b}%
(\omega _{0},a,Z_{s+t_{n}}(\omega _{0}))-\frac{1}{2}\right\} ds,\ \ t\in
\lbrack 0,1],  \label{auxparan}
\end{equation}%
where 
\begin{equation*}
\tilde{m}_{n}=\log \xi (\omega _{0})+\inf \{W_{h+t_{n}}(\omega _{0}):0\leq
h\leq 1\}.
\end{equation*}%
Observe that (\ref{neq:suc}) implies that $\tilde{m}_{n}>c$, for all $n$
large enough. Then Lemma \ref{lem:comp} $(i)$ implies that the explosion
time $T_{\tilde{m}_{n}}^{u}$ of 
\begin{equation*}
u(t)=\tilde{m}_{n}+\int_{0}^{t}\left\{ \tilde{b}(\omega _{0},a,u(s))-\frac{1%
}{2}\right\} ds,\ \ t\geq 0,
\end{equation*}%
is bigger or equal than $1$. Hence 
\begin{equation*}
2\int_{\tilde{m}_{n}}^{\infty }\frac{ds}{2\tilde{b}(\omega _{0},a,s)-1}\geq
1.
\end{equation*}%
Then (\ref{neq:suc}) necessary gives 
\begin{equation*}
\int_{\theta }^{\infty }\frac{ds}{2b(\omega _{0},a,s)-s}=\infty ,\ \ \forall
\theta >e^{c}.
\end{equation*}%
Thus the proof is complete.\hfill
\end{proof}

\bigskip

Now we present other Osgood type criteria.

\begin{proposition}
\label{Propidea}Let $c\in {\mathbb{R}}$ and assume that with probability one
the function $b:{\Omega \times \lbrack }0,\infty )\times {\mathbb{R}}%
\rightarrow {\mathbb{R}}$ satisfy
\newline
$(i)$ $b$ is non-decreasing by components,$\newline
(ii)$ for each $(t,x)\in \lbrack 0,\infty )\times (c,\infty )$, $b(t,x)>0$.$%
\newline
$For almost all $\omega _{0}$ in 
\begin{eqnarray*}
{\tilde{\Omega}}=\{\omega \in \Omega  &:&W_{\cdot }(\omega )\ \text{is
continuous and }b(\omega ,\cdot )\text{ satisfy} \\
&&\text{the above hypotheses, }\xi (\omega )>c\},
\end{eqnarray*}%
if the solution $Y(\omega _{0})$ of (\ref{eq:sigma1}) explodes in finite
time, then%
\begin{equation*}
\int_{\theta }^{\infty }\frac{ds}{b(\omega _{0},a,s)}<\infty ,\ \ \forall
\theta >c,
\end{equation*}%
for some $a>0$.
\end{proposition}

\begin{proof}
By hypothesis $T_{\xi }^{Y}(\omega _{0})<\infty $. The continuity of $%
W_{\cdot }(\omega _{0})$ implies that 
\begin{equation*}
m=\inf \{e^{W_{s}(\omega _{0})-\frac{s}{2}}:s\in \lbrack 0,T_{\xi
}^{Y}(\omega _{0})]\}>0.
\end{equation*}%
Since $R(\omega _{0})$ explodes to $+\infty $, then there exists a $%
0<T<T_{\xi }^{Y}(\omega _{0})$ such that 
\begin{equation}
R_{s}(\omega _{0})>\frac{c}{m},\ \ T\leq s<T_{\xi }^{Y}(\omega _{0}).
\label{cetsR}
\end{equation}%
From (\ref{auxva}) we see that 
\begin{multline*}
R_{T+t}(\omega _{0})=R_{T}(\omega _{0})+\int_{0}^{t}e^{-W_{T+s}(\omega _{0})+%
\frac{T+s}{2}}b(\omega _{0},T+s,e^{W_{T+s}(\omega _{0})-\frac{T+s}{2}}R_{T+s})ds, \\
0\leq t<T_{\xi }^{Y}(\omega _{0})-T.
\end{multline*}%
The condition (\ref{cetsR}) implies that 
\begin{equation}
MR_{T+s}(\omega _{0})\geq e^{W_{T+s}(\omega _{0})-\frac{T+s}{2}}R_{T+s}\geq
mR_{T+s}>c,\ \ 0\leq s<T_{\xi }^{Y}(\omega _{0})-T,  \label{audR}
\end{equation}%
where
\begin{equation*}
M=R_{T}+\exp \left( \sup_{t\in[0, T_{\xi}^{Y}(\omega _{0})]} |W_{t}(\omega _{0})| +\frac{T_{\xi }^{Y}(\omega _{0})}{2}\right) +c+1.
\end{equation*}%
Now, using the hypothesis $(i)$ we get
\begin{equation*}
R_{T+t}(\omega _{0})\leq M+\int_{0}^{t}Mb(\omega _{0},T_{\xi }^{Y}(\omega
_{0}),MR_{T+s}(\omega _{0}))ds,\ \ 0\leq t<T_{\xi }^{Y}(\omega _{0})-T.
\end{equation*}%
Let us define $y(t)=MR_{T+t}(\omega _{0})$, then the previous inequality
leads to%
\begin{equation*}
y(t)\leq M^{2}+\int_{0}^{t}M^{2}b(\omega _{0},T_{\xi }^{Y}(\omega _{0}),y(s))ds,\ \
0\leq t<T_{\xi }^{Y}(\omega _{0})-T.
\end{equation*}%
We consider the integral equation%
\begin{equation*}
x(t)=M^{2}+\int_{0}^{t}M^{2}b(\omega _{0},T_{\xi }^{Y}(\omega
_{0}),x(s))ds,\ \ t\geq 0.
\end{equation*}%
It is clear that $M^{2}>c$ and (\ref{audR}) yields $y>c$ on $[0,T_{\xi
}^{Y}(\omega _{0})-T)$, hence by Lemma \ref{lem:comp} $(ii)$ we deduce that $x(t)\geq
MR_{T+t}(\omega _{0})$, for all $0\leq t<T_{\xi }^{Y}(\omega _{0})-T$, therefore by
Osgood criterion's we obtain%
\begin{equation*}
\int_{M^{2}}^{\infty }\frac{ds}{M^{2}b(\omega _{0},T_{\xi }^{Y}(\omega
_{0}),s)}<\infty .
\end{equation*}%
The result follows from the continuity of $b$ and hypothesis $(ii)$.\hfill 
\end{proof}

\bigskip

\begin{example}
Consider the equation%
\begin{equation}
Y_{t}=1+\frac{1}{2}\int_{0}^{t}Y_{s}^{2}ds+\int_{0}^{t}Y_{s}dW_{s},\ \ t\geq
0.  \label{exem}
\end{equation}%
Proceeding as in (\ref{seuspeje}) we deduce that $Y>0$. Therefore we can use
Feller's test (Theorem \ref{feller}) to see the explosive behavior of $Y$ in 
$(0,\infty ]$. In this case, by equation (\ref{funfeller}) we see that%
\begin{eqnarray*}
p(0) &=&-\int_{0}^{1}\exp \left( \int_{s}^{1}dr\right) ds \\
&=&1-e>-\infty 
\end{eqnarray*}%
and by Fubini's theorem equation (\ref{defvF}) can be written as%
\begin{eqnarray*}
v(0) &=&2\int_{0}^{1}\int_{y}^{1}\frac{\exp \left( z-y\right) }{z^{2}}dzdy \\
&\geq &2\int_{0}^{1}\int_{y}^{1}\frac{1}{z^{2}}dzdy \\
&=&2\int_{0}^{1}\frac{1}{z}dz=\infty .
\end{eqnarray*}%
Thus the solution $Y$ of (\ref{exem}) does not blow up in finite time with
positive probability. However note that 
\begin{equation*}
\int_{\theta }^{\infty }\frac{dr}{r^{2}}<\infty ,\ \ \forall \theta >0,
\end{equation*}%
hence we do not have the converse of Proposition \ref{Propidea}. As we shall
see in Proposition \ref{noconver} the reason of this singularity is that 
\begin{equation}
\int_{0}^{\infty }\frac{dr}{r^{2}}=\infty .  \label{caseinver}
\end{equation}
\end{example}

\bigskip

We have the converse of Proposition \ref{Propidea} if
the corresponding integral is divergent, that is nothing similar to case (%
\ref{caseinver}).

\begin{proposition}
\label{noconver}Assume that with probability 1 the function $\tilde{b}:{%
\Omega \times \lbrack }0,\infty )\allowbreak \times \allowbreak {\mathbb{R}}%
\rightarrow {\mathbb{R}}$ satisfy
\newline
$(i)$ $\tilde{b}$ is non-decreasing by components,$\newline
(ii)$ for each $(t,x)\in \lbrack 0,\infty )\times (0,\infty )$, $b(t,x)>0$.$%
\newline
$Then the solution $Y(\omega _{0})$ of (\ref{eq:sigma1}) explodes in finite
time if 
\begin{equation*}
\int_{0}^{\infty }\frac{ds}{b(\omega _{0},a,s)}<\infty 
\end{equation*}%
for some $a>0$. Here $\omega _{0}$ is in the set, of probability 1, 
\begin{eqnarray*}
{\tilde{\Omega}}=\{\omega \in \Omega  &:&W_{\cdot }(\omega )\ \text{is
continuous and }\tilde{b}(\omega ,\cdot )\text{, }b(\omega ,\cdot )\text{
satisfy} \\
&&\text{the above hypotheses, }\xi (\omega )>0\}.
\end{eqnarray*}
\end{proposition}

\begin{proof}
Suppose that $T_{\xi }^{Y}(\omega _{0})=\infty $. As in (\ref{auxparan}) we
obtain, for $a>0$, 
\begin{equation*}
Z_{t+a}(\omega _{0})\geq \log \xi (\omega _{0})+W_{t+a}(\omega _{0})-\frac{%
t+a}{2}+\int_{0}^{t}\tilde{b}(\omega _{0},a,Z_{s+a}(\omega _{0}))ds,\ \
t\geq 0.
\end{equation*}%
Renaming $X_{t}=Z_{t+a}$ we obtain%
\begin{equation*}
X_{t}(\omega _{0})\geq m_{n}+\int_{0}^{t}\tilde{b}(\omega
_{0},a,X_{s}(\omega _{0}))ds,\ \ t\in \lbrack 0,n],
\end{equation*}%
where%
\begin{equation*}
m_{n}=\log \xi (\omega _{0})-\sup_{t\in \lbrack 0,n]}|W_{t+a}(\omega _{0})|-%
\frac{n+a}{2}.
\end{equation*}%
In a similar fashion as in previous results we take into account the
equation 
\begin{equation*}
x(t)={m_{n}}+\int_{0}^{t}\tilde{b}(\omega _{0},a,x(s))ds,\ \ t\geq 0.
\end{equation*}%
From the comparison Lemma \ref{lem:comp} $(i)$ and Osgood's criterion (see
Proposition \ref{osgood}) we can establish the inequality 
\begin{equation*}
\int_{0}^{\infty }\frac{ds}{b(\omega _{0},a,s)}\geq \int_{\exp
(m_{n})}^{\infty }\frac{ds}{b(\omega _{0},a,s)}\geq n.
\end{equation*}%
The result is obtained by letting $n\rightarrow \infty $.\hfill 
\end{proof}

\bigskip

\subsection{Main results}

Now we are ready to state the main results of this article.

\begin{theorem}
\label{the:nec} Assume that with probability one\newline
$(i)$ $b$ is non-decreasing by components,\newline
$(ii)$ for each $(t,x)\in \lbrack 0,\infty )\times (0,\infty )$, $b(t,x)>0.$%
\newline
Let $X$ be the solution of equation (\ref{eqpmb}) and 
\begin{eqnarray*}
{\tilde{\Omega}}=\{\omega \in \Omega  &:&W_{\cdot }(\omega )\ \text{is
continuous and }b(\omega ,\cdot )\text{ satisfy} \\
&&\text{the above hypotheses, }\xi (\omega )>0\}.
\end{eqnarray*}%
For almost all $\omega _{0}$ in $\tilde{\Omega}$, if $X_{\cdot }(\omega _{0})$ explodes in finite time then%
\begin{equation*}
\int_{\theta }^{\infty }\frac{ds}{b(\omega _{0},a,s)}<\infty ,\ \ \forall
\theta >0,
\end{equation*}%
for some $a>0.$
\end{theorem}

\begin{proof}
Set 
\begin{equation}
g(t)=\exp \left( -\int_{0}^{t}\sigma (s)dW_{s}+\frac{1}{2}\int_{0}^{t}\sigma
^{2}(s)ds\right) ,\ \ t\geq 0.  \label{fung}
\end{equation}%
So, using It\^{o}'s formula, we have%
\begin{equation}
Y_{t}=\xi +\int_{0}^{t}g(s)b(s,f(s)Y_{s})ds,\ \ t\geq 0,  \label{edem}
\end{equation}%
where 
\begin{equation*}
Y_{t}=g(t)X_{t}\text{ \ and \ }f(t)=\frac{1}{g(t)},\ \ t\geq 0.
\end{equation*}%
Consequently, for $\omega _{0}\in {\tilde{\Omega}}$ such that satisfies (\ref{edem}), the continuity of $g$
and $b(\omega _{0},\cdot )$ imply that (\ref{edem}) can be written as%
\begin{eqnarray}
y^{\prime }(t) &=&g(\omega _{0},t)b(\omega _{0},t,f(\omega _{0},t)y(t)),\ \
t>0,  \label{edo} \\
y(0) &=&\xi (\omega _{0}),  \notag
\end{eqnarray}%
where $y(t)=Y_{t}(\omega _{0})$, $t\geq 0$. Since $b\geq 0$ then $f(\omega
_{0},t)y(t)\geq f(\omega _{0},t)\xi >0$. Therefore (\ref{edo}) and
hypothesis (ii) turns out%
\begin{eqnarray}
\int_{\xi (\omega _{0})}^{Y_{t}(\omega _{0})}\frac{ds}{b(\omega
_{0},y^{-1}(s),sf(\omega _{0},y^{-1}(s)))} &=&\int_{0}^{t}g(\omega _{0},s)ds
\notag \\
 &:=&G(\omega _{0},t).  \label{expliy}
\end{eqnarray}%
Suppose now that $T_{\xi (\omega _{0})}^{Y(\omega _{0})}<\infty $, where $
T_{\xi (\omega _{0})}^{Y(\omega _{0})}$ is defined as in the Proposition \ref%
{osgood}. Hence, $y^{-1}(t)<T_{\xi (\omega _{0})}^{Y(\omega _{0})}$, $t\geq
\xi (\omega _{0})$, and therefore hypothesis\textbf{\ }$(i)$ implies%
\begin{eqnarray*}
\int_{\xi (\omega _{0})}^{\infty }\frac{ds}{b(\omega _{0},T_{\xi (\omega
_{0})}^{Y(\omega _{0})},sM)} &\leq &\int_{\xi (\omega _{0})}^{Y_{T_{\xi
(\omega _{0})}^{Y(\omega _{0})}}}\frac{ds}{b(\omega _{0},y^{-1}(s),sf(\omega
_{0},y^{-1}(s)))}, \\
&=&G(\omega _{0},T_{\xi (\omega _{0})}^{Y(\omega _{0})})<\infty ,
\end{eqnarray*}%
with 
\begin{equation*}
M=\sup \{f(\omega _{0},y^{-1}(s)):s\geq \xi (\omega _{0})\}\leq \sup
\{f(\omega _{0},r):r\in \lbrack 0,T_{\xi (\omega _{0})}^{Y(\omega _{0})}]\}.
\end{equation*}%
Thus, the proof is complete \hfill
\end{proof}

\bigskip

In the remainder of this paper we will need the following notation 
\begin{equation*}
\Lambda (t)=\int_{0}^{t}\sigma ^{2}(s)ds,\ \ t\geq 0.
\end{equation*}

The following two results are in certain sense the converse of Theorem \ref%
{the:nec}.

\begin{theorem}
Let $X$ be the solution of (\ref{eqpmb}). Assume that hypotheses of Theorem \ref{the:nec} are true.  Let $\omega_{0}\in {\tilde{\Omega}}$ be such that $\Lambda (\omega _{0},\infty )<\infty$ and $X_{t}(\omega _{0})$ is finite for all $t\geq 0$. Then, 
\begin{equation*}
\int_{\theta }^{\infty }\frac{ds}{b(\omega _{0},a,s)}=\infty ,\ \ \forall
\theta >0,
\end{equation*}%
for all $a>0$.
\end{theorem}

\begin{remark}
\label{recvcfint}Theorem 3.4.9 in \cite{Du}, implies that 
\begin{equation*}
\left\{ \omega \in \Omega :\int_{0}^{\cdot }\sigma (s)dW_{s}\ \text{is
bounded on }\mathbb{R}_{+}\right\}
\end{equation*}%
coincides with the set $\{\omega \in \Omega :\Lambda (\omega ,\infty
)<\infty \}$ by redefining $\sigma $ on a set of probability zero.
\end{remark}

\begin{proof}
Let $\omega _{0}$ be as in the statement of the theorem. Then (\ref{fung})
and (\ref{expliy}) lead us to%
\begin{eqnarray}
G(\omega _{0},\infty ) &\geq &\int_{0}^{\infty }\exp \left( -\left(
\int_{0}^{s}\sigma (r)dW_{r}\right) (\omega _{0})\right) ds  \notag \\
&\geq &\int_{0}^{\infty }\exp \left( -\sup_{t\geq 0}\left\vert \left(
\int_{0}^{t}\sigma (r)dW_{r}\right) (\omega _{0})\right\vert \right)
ds=\infty ,  \label{ginfty}
\end{eqnarray}%
where we have used the Remark \ref{recvcfint} in the last equality. Inasmuch
as $b\geq 0$ we see that $Y(\omega _{0})$ is increasing, then $%
\lim_{t\rightarrow \infty }Y_{t}(\omega _{0})$ exists. On the other hand, (%
\ref{expliy}) and (\ref{ginfty}) implies%
\begin{equation}
\int_{\xi (\omega _{0})}^{\underset{t\rightarrow \infty }{\lim }Y_{t}(\omega
_{0})}\frac{ds}{b(\omega _{0},y^{-1}(s),sf(\omega _{0},y^{-1}(s)))}=\infty .
\label{inqinlif}
\end{equation}%
Since $b(\omega _{0},\cdot )$ is continuous and $b(\omega _{0},\cdot )>0$ on 
$[0,\infty )\times (0,\infty )$ we can deduce that $\lim_{t\rightarrow
\infty }Y_{t}(\omega _{0})=\infty $. Let $a>0$, by (\ref{inqinlif}) one
obtains%
\begin{equation*}
\infty =\int_{Y_{a}(\omega _{0})}^{\infty }\frac{ds}{b(\omega
_{0},y^{-1}(s),sf(\omega _{0},y^{-1}(s)))}\leq \int_{Y_{a}(\omega
_{0})}^{\infty }\frac{ds}{b(\omega _{0},a,sm))},
\end{equation*}%
where 
\begin{eqnarray*}
m=\inf \{f(\omega _{0},r):r\geq 0\} &\geq &\exp \left( -\frac{1}{2}\Lambda
(\omega _{0},\infty )\right) \\
&&\times \exp \left( -\sup_{t\geq 0}\left\vert \left( \int_{0}^{t}\sigma
(r)dW_{r}\right) (\omega _{0})\right\vert \right) .
\end{eqnarray*}%
Using again the Remark \ref{recvcfint} we deduce that $m>0$, from which
allows us to conclude the result.\hfill
\end{proof}

\begin{theorem}
Assume that with probability one \newline
$(i)$ $\sigma ^{2}>0$ in $(0,\infty )$ and $\Lambda (\infty )=\infty $, 
\newline
$(ii)$ the function $\breve{b}:{\Omega \times \lbrack }0,\infty )\times {%
\mathbb{R}}\rightarrow {\mathbb{R}}$, defined as%
\begin{equation}
\breve{b}(\omega ,t,x)=\frac{b(\omega ,\Lambda ^{-1}(t),e^{x})}{\sigma
^{2}(\Lambda ^{-1}(t))e^{x}},  \label{nbtilde}
\end{equation}%
is non-decreasing by components, \newline
$(iii)$ for each $(t,x)\in \lbrack 0,\infty )\times (0,\infty )$, $b(t,x)>0$%
. \newline
For almost all $\omega _{0}$ in 
\begin{eqnarray*}
{\tilde{\Omega}}=\{\omega \in \Omega  &:&W_{\cdot }(\omega )\ \text{is
continuous and }\breve{b}(\omega ,\cdot ),\ b(\omega ,\cdot )\text{ satisfy%
} \\
&&\text{the above hypotheses, }\xi (\omega )>0\}
\end{eqnarray*}%
the solution $X_{\cdot }(\omega _{0})$ of (\ref{eqpmb}) explodes in finite
time if 
\begin{equation*}
\int_{0}^{\infty }\frac{ds}{b(\omega _{0},a,s)}<\infty 
\end{equation*}%
for some $a>0$.
\end{theorem}

\begin{proof}
From Theorem 3.4.4 in \cite{Du} we know that there exists a Brownian motion $%
{\tilde{B}}=\{{\tilde{B}}_{t}:t\geq 0\}$ such that 
\begin{equation*}
\int_{0}^{t}\sigma (s)dW_{s}={\tilde{B}}_{\Lambda (t)},\ \ t\geq 0.
\end{equation*}%
This leads us to write (\ref{edem}) as%
\begin{equation*}
Y_{t}=\xi +\int_{0}^{t}e^{-{\tilde{B}}_{\Lambda (s)}+\frac{1}{2}\Lambda
(s)}b(s,e^{{\tilde{B}}_{\Lambda (s)}-\frac{1}{2}\Lambda (s)}Y_{s})ds,\ \
t\geq 0.
\end{equation*}%
Moreover, making the change of variable $u=\Lambda (s)$ and setting $%
Z_{t}=Y_{\Lambda ^{-1}(t)}$ we consider the equation%
\begin{equation*}
Z_{t}=\xi +\int_{0}^{t}\frac{e^{-{\tilde{B}}_{s}+\frac{1}{2}s}}{\sigma
^{2}(\Lambda ^{-1}(s))}\ b(\Lambda ^{-1}(s),e^{{\tilde{B}}_{s}-\frac{1}{2}%
s}Z_{s})ds,\ \ t\geq 0.
\end{equation*}%
Finally, if ${\tilde{Z}}_{t}=Z_{t}e^{{\tilde{B}}_{t}-t/2}$, by It\^{o}'s
formula we have%
\begin{equation*}
{\tilde{Z}}_{t}=\xi +\int_{0}^{t}\frac{b(\Lambda ^{-1}(s),{\tilde{Z}}_{s})}{%
\sigma ^{2}(\Lambda ^{-1}(s))}ds+\int_{0}^{t}{\tilde{Z}}_{s}d{\tilde{B}}%
_{s},\ \ t\geq 0,
\end{equation*}%
and the result follows from Proposition \ref{noconver}, because $b$ and $\breve{b}$ meet the respective assumptions of such proposition.\hfill 
\end{proof}

\begin{remark}
Let $c\geq 0$ and suppose that with probability one the function $\breve{b}$, defined in (\ref{nbtilde}), satisfies hypothesis $(i)-(iii)$ in Theorem %
\ref{trespuntodos}. Then, for all $\omega _{0}\in {\tilde{\Omega}}$ the
solution $X_{\cdot }(\omega _{0})$ of equation (\ref{eqpmb}) explodes in
finite time if and only if 
\begin{equation*}
\int_{\theta }^{\infty }\frac{ds}{2b(\omega _{0},a,s)-\sigma^{2}(a)s}<\infty ,\
\ \forall \theta >e^{c},
\end{equation*}%
for some constant $a>0$.
\end{remark}

\bigskip

\noindent \textit{Acknowledgment:} The authors thanks Universidad Aut\'{o}%
noma de Aguascalientes and CINVESTAV-IPN for their hospitality and
economical support.

\end{document}